\documentclass{amsart}

\usepackage{amssymb} 
\usepackage{comment}

\DeclareMathOperator{\tr}{tr}
\DeclareMathOperator{\Sp}{Sp}
\DeclareMathOperator{\GSp}{GSp}

\DeclareMathOperator{\End}{End}

\DeclareMathOperator{\GL}{GL}

\DeclareMathOperator{\Ker}{Ker}

\DeclareMathOperator{\Res}{Res}

\DeclareMathOperator{\Gal}{Gal}

\DeclareMathOperator{\lcm}{lcm}

\newcommand{\Q}{{\mathbb Q}}
\newcommand{\Z}{{\mathbb Z}}
\newcommand{\F}{{\mathbb F}}
\newcommand{\OO}{\mathcal{O}}

\def\F{{\ensuremath{\mathbb{F}}}}

\begin{document}

\newtheorem{thm}{Theorem}
\newtheorem{lem}{Lemma}[section]
\newtheorem{prop}[lem]{Proposition}
\newtheorem{cor}[lem]{Corollary}
\newtheorem{conj}{Conjecture}

\theoremstyle{definition}
\newtheorem*{remark}{Remark}
\newtheorem*{ex}{Example}

\theoremstyle{remark}
\newtheorem{ack}{Acknowledgement}

\title[Residual Representations]{Residual Representations of Semistable\\
Principally Polarized Abelian Varieties}
\author{Samuele Anni, Pedro Lemos and Samir Siksek}

\address{Mathematics Institute\\
	University of Warwick\\
	Coventry\\
	CV4 7AL \\
	United Kingdom}
\email{samuele.anni@gmail.com}
\email{lemos.pj@gmail.com}
\email{samir.siksek@gmail.com}

\date{\today}
\thanks{The first-named and third-named authors are supported by EPSRC Programme Grant 
\lq LMF: L-Functions and Modular Forms\rq\  EP/K034383/1.
}


\begin{abstract}
Let $A/\Q$ be a semistable principally 
polarized abelian variety of dimension $d \ge 1$.
Let $\ell$ be a prime and let $\overline{\rho}_{A,\ell}\colon G_\Q \rightarrow \GSp_{2d}(\F_\ell)$ be the representation
giving the action of $G_\Q :=\Gal(\overline{\Q}/\Q)$
on the $\ell$-torsion group $A[\ell]$. 
We show that if 
$\ell \ge \max(5,d+2)$, 
 and if image of $\overline{\rho}_{A,\ell}$
contains a transvection then $\overline{\rho}_{A,\ell}$ is either
reducible or surjective.

With the help of this we study surjectivity of $\overline{\rho}_{A,\ell}$
for semistable polarized abelian threefolds, 
and give an example of a genus $3$ hyperelliptic curve $C/\Q$
such that $\overline{\rho}_{J,\ell}$ is surjective for all primes $\ell \ge 3$,
where $J$ is the Jacobian of $C$.
\end{abstract}
\maketitle

\section{Introduction}
Let $A$ be a principally polarized abelian variety of dimension $d$ defined over $\Q$.
Let $\ell$ be
a prime 
and write $A[\ell]$ for the $\ell$-torsion subgroup of 
$A(\overline{\Q})$. This is a $2d$-dimensional $\F_\ell$-vector space,
as well as a $G_\Q$-module, where $G_\Q:=\Gal(\overline{\Q}/\Q)$.
The polarization induces the mod $\ell$
Weil pairing on $A[\ell]$, which is a bilinear, alternating,
non-degenerate
pairing 
\[
\langle~,~\rangle \; \colon \; A[\ell] \times A[\ell] \rightarrow \F_\ell(1)
\]
that is Galois equivariant. The latter property means
$\langle \sigma v, \sigma v^\prime \rangle=\chi(\sigma) \langle v,v^\prime \rangle$
for all $\sigma \in G_\Q$, and $v$, $v^\prime \in A[\ell]$ where
$\chi \colon  G_\Q \rightarrow \F_\ell^\times$ is the mod $\ell$ cyclotomic character.
In particular,
the space $(A[\ell], \langle~,~\rangle)$ is a symplectic $\F_\ell$-vector
space of dimension $2d$. We obtain a representation
\[
\overline{\rho}_{A,\ell} \; \colon \; G_\Q \rightarrow \GSp(A[\ell], \langle~,~\rangle)
\cong \GSp_{2d}(\F_\ell).
\]
The study of images of representations $\overline{\rho}_{A,\ell}$
has received much attention recently (e.g.\ \cite{dieulefait4}, \cite{Hall11}, \cite{Sara14}). 
This study is largely motivated by the following remarkable result of Serre.
\begin{thm}[Serre, {\cite[Theorem~3]{serreIV}}] 
Let $A$ be a principally polarized abelian variety of dimension $d$, 
defined over $\Q$. 
Assume that $d = 2$, $6$ or $d$ is odd and furthermore assume that 
$\End_{\overline{\Q}}(A)=\Z$. Then there exists a bound $B_{A}$ such that 
for all primes $\ell > B_{A}$ the representation $\overline{\rho}_{A,\ell}$
is surjective.
\end{thm}
For explicit (though large) estimates for the constant $B_A$ see \cite{Lombardo}.
The conclusion of the theorem is known to be false for general $d$; a counterexample
is constructed by Mumford \cite{Mumford69} for $d=4$.
The following is a tantalizing open question: 
\emph{given $d$ as in the theorem, is there a uniform bound $B_d$ depending only
on $d$, such that for all principally polarized abelian 
varieties $A$ 
over $\Q$ of dimension $d$ with $\End_{\overline{\Q}}(A)=\Z$,
 and all $\ell>B_d$,
the representation $\overline{\rho}_{A,\ell}$ is surjective?}
For elliptic curves an affirmative answer is expected,
and this is known as Serre's Uniformity Question, which
is still an open problem. Serre's Uniformity Question is
much easier for semistable elliptic curves. Indeed,
Serre \cite[Proposition 21]{serre72} shows that if $E/\Q$
is a semistable elliptic curve, and $\ell \ge 7$ is prime,
then $\overline{\rho}_{E,\ell}$ is either surjective or
reducible. It immediately follows 
from Mazur's classification \cite{Mazur78}
of isogenies of elliptic curves over $\Q$ that $\overline{\rho}_{E,\ell}$
is surjective for $\ell \ge 11$. 
It is natural to ask
if a result similar to Serre's \cite[Proposition 21]{serre72} 
can be established for semistable principally
polarized abelian varieties. For now, efforts to prove such
a theorem are hampered by the absence of a satisfactory
classification of maximal subgroups of $\GSp_{2d}(\F_\ell)$ (indeed,
Serre's result for semistable elliptic curves makes
use of Dickson's classification of 
maximal subgroups of $\GL_2(\F_\ell)=\GSp_2(\F_\ell)$).
There is however a beautiful classification
due to Arias-de-Reyna, Dieulefait and Wiese (see Theorem~\ref{thm:ADW}
below)
of
subgroups of $\GSp_{2d}(\F_\ell)$ containing a transvection.
A \emph{transvection} is a unipotent element $\sigma$ such
$\sigma-I$ has rank $1$.
The main result of our paper, building on the classification
of Arias-de-Reyna, Dieulefait and Wiese, is the following theorem.
\begin{thm}\label{thm:semistable}
Let $A$ be a semistable principally polarized
abelian variety of dimension $d \ge 1$ over $\Q$ and let $\ell \ge \max(5,d+2)$ 
be prime. Suppose the image of
$\overline{\rho}_{A,\ell} \; \colon \; G_\Q \rightarrow  \GSp_{2d}(\F_\ell)$
contains a transvection. Then $\overline{\rho}_{A,\ell}$ is either
reducible or surjective.
\end{thm}

The proof of Theorem~\ref{thm:semistable} is given in Section~\ref{sec:semiproof}.
We mention a well-known pair of sufficient conditions (see for example
\cite[Section 2]{Hall11}) for 
the image of $\overline{\rho}_{A,\ell}$
to contain a transvection.
Let $q \ne \ell$ be a prime and suppose that the following two conditions are satisfied:
\begin{itemize}
\item The special fibre of the N\'{e}ron
model for $A$ at $q$ has toric dimension $1$; 
\item 
$\ell \nmid \#\Phi_q$,
where $\Phi_q$ is the group of connected components of the special
fibre of the N\'{e}ron model at $q$. 
\end{itemize} 
Then the image of $\overline{\rho}_{A,\ell}$ contains
a transvection.
Now let $C/\Q$ be a hyperelliptic
curve of genus $d$, given by a model $y^2=f(x)$ where $f \in \Z[x]$
is a squarefree polynomial. Let $p$ be an odd prime not dividing
the leading coefficient of $f$ such that $f$ modulo $p$
has one root in $\overline{\F}_p$ having multiplicity precisely $2$,
with all other roots simple.
Then the N\'{e}ron model for the Jacobian $J(C)$ (which is a principally polarized
abelian variety) at $p$ has toric dimension $1$.

\bigskip

The remainder of the paper is concerned with semistable
principally polarized abelian threefolds $A/\Q$
which
possess a prime $q$ such that the special fibre of the N\'{e}ron
model for $A$ at $q$ has toric dimension $1$.
Building on Theorem~\ref{thm:semistable}, we give in
Section~\ref{sec:bound} a  practical method
which should in most cases produce an explicit (and small) bound $B$ (depending on $A$)
such that for $\ell \ge B$, the representation $\overline{\rho}_{A,\ell}$ 
is surjective. This method is inspired by the paper of 
Dieulefait \cite{dieulefait4}
which solves the corresponding problem for abelian surfaces. 
Our method is not always guaranteed to succeed, but we expect it to succeed
if $\End_{\overline{\Q}}(A) = \Z$.

\bigskip 

It is well known that $\GSp_{2d}(\F_\ell)$ is a Galois group over
 $\Q$ for all $d \ge 1$ and all primes $\ell\ge 3$
(see for example \cite[Remark 2.5]{Sara14}). The proof of this statement in fact shows the existence
of a genus $d$ hyperelliptic curve $C/\Q$ such that $\overline{\rho}_{J,\ell}$ is surjective,
where $J$ is the Jacobian of $C$.
The argument however relies on Hilbert's Irreducibility Theorem, and so does not produce
an explicit equation for $C$. In \cite{Sara14}, a genus $3$ hyperelliptic curve $C/\Q$ is given
so that $\overline{\rho}_{J,\ell}$ is surjective for all $11 \le \ell < 5 \times 10^5$.
As a corollary to our method for producing the bound $B$ mentioned
above, we prove the following.
\begin{cor}\label{cor:surjective}
Let $C/\Q$ be the following genus $3$ hyperelliptic curve,
\begin{equation}\label{eqn:gen3}
C \; : \;
y^2 + (x^4 + x^3 + x + 1)y = x^6 + x^5 .
\end{equation}
and write $J$ for its Jacobian.
Let $\ell \ge 3$ be a prime. Then $\overline{\rho}_{J,\ell}(G_\Q)=\GSp_6(\F_\ell)$.
\end{cor}
Recently (and independently), Zywina \cite{Zywina} 
gives a genus $3$ plane quartic curve over $\Q$ for which 
he proves that the 
Galois action on the torsion subgroup of the Jacobian $J$
is maximal (in other words,
$\rho_J \colon G_\Q \rightarrow \GSp_6(\hat{\Z})$ 
is surjective).

We are grateful to Tim Dokchitser for providing us with a list
of genus $3$ hyperelliptic curves with small Jacobian conductors, from which
we chose the curve $C$ in Corollary~\ref{cor:surjective}. We thank Jeroen Sijsling
for helpful remarks on an earlier version of this paper.
We are indebted to the referees for their careful reading of the paper
and for many helpful remarks.

\section{Proof of Theorem~\ref{thm:semistable}}\label{sec:semiproof}

We shall make use of the following classification of subgroups of 
$\GSp_{2d}(\F_\ell)$ containing a transvection, due to 
Arias-de-Reyna, Dieulefait and Wiese.
\begin{thm}[Arias-de-Reyna, Dieulefait and Wiese, {\cite{disawi}}]\label{thm:ADW}
Let $\ell \ge 5$ be a prime and let
$V$ a symplectic $\F_\ell$-vector space of dimension $2d$.
Any subgroup $G$ of $\GSp (V)$ which contains
a transvection satisfies one 
of the following assertions:
\begin{enumerate}
\item[(i)] There is a non-trivial proper $G$-stable subspace $W \subset V$.
\item[(ii)] There are non-singular symplectic subspaces $V_i \subset V$
with $i=1,\dotsc,h$, of dimension $2m<2d$ 
and a homomorphism $\phi \; \colon \; G \rightarrow S_h$
such that $V=\oplus_{i=1}^h V_i$ and $\sigma(V_i)=V_{\phi(\sigma)(i)}$
for $\sigma \in G$ and $1 \le i \le h$. Moreover, 
$\phi(G)$ is a transitive subgroup of $S_h$. 
\item[(iii)] $\Sp(V) \subseteq G$.
\end{enumerate}
\end{thm}
We shall apply Theorem~\ref{thm:ADW} to $G=\overline{\rho}_{A,\ell}(G_\Q)$
where $A$ and $\ell$ are as in the statement of Theorem~\ref{thm:semistable}.
It follows from the surjectivity of the mod $\ell$ cyclotomic character 
$\chi \colon G_\Q \rightarrow \F_\ell^\times$
that if $\Sp_{2d}(\F_\ell) \subseteq G$ then $G=\GSp_{2d}(\F_\ell)$.

\bigskip

Throughout we write $I_\ell$ for the inertia at $\ell$
subgroup of $G_\Q$.
We shall also make use of the following theorem of Raynaud.
\begin{thm}[Raynaud, {\cite{Raynaud}}]
Let $A$ be an abelian variety over $\Q$.
Let $\ell$ be a prime of semistable reduction
for $A$.
Regard $A[\ell]$
as an $I_\ell$-module and let $V$ be a Jordan-H\"{o}lder
factor of dimension $n$ over $\F_\ell$. 
Let $\psi_n \; \colon \; I_\ell \rightarrow \F_{\ell^n}^\times$
be a fundamental character of level $n$.
Then $V$
has the structure of a $1$-dimensional $\F_{\ell^n}$-vector space
and the action of $I_\ell$ on it is given by a character 
$\varpi \; \colon \; I_\ell \rightarrow \F_{\ell^n}^\times$,
where $\varpi=\psi_n^{\sum_{i=0}^{n-1} a_i \ell^i}$
with $a_i=0$ or $1$.
\end{thm}

We shall make use of the following elementary lemma in the proof 
of Theorem~\ref{thm:semistable}.
\begin{lem}\label{lem:cyclic}
Let $k$ be an algebraically closed field and $V \ne 0$ be a finite dimensional
vector space over $k$. Let $T \colon V \rightarrow V$ be a $k$-linear map,
and suppose $V=\oplus_{i=1}^{r} V_i$ where $T(V_i)=V_{i+1}$ (the indices
considered modulo $r$). Let $\alpha$ be an eigenvalue of $T$.
Then $\zeta \alpha$ is also an eigenvalue of $T$ for every $\zeta$
in $k$ satisfying $\zeta^r=1$.
\end{lem}
\begin{proof}
Let $v$ be an eigenvector corresponding to $\alpha$ and write
$v=\sum_{i=1}^r v_i$ with $v_i \in V_i$. Then
$T(v_i)=\alpha v_{i+1}$. Let $v^\prime=\sum_{i=1}^r \zeta^{-i} v_i$.
Then $T (v^\prime)=\sum_{i=1}^r \zeta^{-i} \alpha v_{i+1}=
\zeta \alpha \sum_{i=1}^r \zeta^{-i-1} v_{i+1}=\zeta \alpha v^\prime$
showing that $\zeta \alpha$ is indeed an eigenvalue.
\end{proof}

\bigskip

\begin{proof}[Proof of Theorem~\ref{thm:semistable}]
Denote $\overline{\rho}=\overline{\rho}_{A,\ell}$.
Let $G=\overline{\rho}(G_\Q) \subseteq \GSp_{2d}(\F_\ell)$ and consider
the action of $G$ on the symplectic vector space $V=A[\ell]$. Since $G$
contains a transvection we may apply Theorem~\ref{thm:ADW}.
To prove the theorem, it is sufficient to show that case (ii) of Theorem~\ref{thm:ADW} 
does not arise. Suppose otherwise. 
Then we can write $V=\oplus_{i=1}^h V_i$ where
$V_i$ are non-singular symplectic subspaces of dimension $2m<2d$,
 and there is some
$\phi \; \colon \; G \rightarrow S_h$ with transitive image
such that $\sigma(V_i)=V_{\phi(\sigma)(i)}$. Let $\pi=\phi \circ \overline{\rho} \colon G_\Q \rightarrow S_h$.
Let $H=\Ker(\pi)$. Then $H=G_K$ for some number field $K/\Q$.
Moreover, $\overline{\rho}\vert_{G_K}$ is reducible as the $V_i$ are
stable under the action of $G_K$. We shall
show that the extension $K/\Q$ is unramified at the finite places,
and thus $K$ has discriminant $\pm 1$.
It then follows by a famous theorem of Hermite that $K=\Q$, showing that $\pi$ is trivial
and contradicting the fact that $\phi(G)=\pi(G_\Q)$ is transitive.

First let $p \ne \ell$ be a prime. As $A$ is semistable,
$I_p$ acts unipotently on $V$. Thus $\overline{\rho}(\sigma)$
has $\ell$-power order for $\sigma \in I_p$. However, the
order of $\overline{\rho}(\sigma)$ is divisible by the order
of $\pi(\sigma)$ which in turn divides $h!$. As $h=2d/2m \le d<\ell$,
we see that $\pi(\sigma)=1$. Thus $K/\Q$ is unramified at $p$.

Next, consider $\sigma \in I_\ell^w$, the wild subgroup of $I_\ell$.
As $I_\ell^w$ is a pro-$\ell$ group,
 $\overline{\rho}(\sigma)$ has $\ell$-power order, and we see that
$\pi(\sigma)=1$ as above. Finally, let $\sigma \in I_\ell$
be an element whose image in the tame inertia group $I_\ell^t=I_\ell/I_\ell^w$
is a topological generator. Reorder $V_1,\dotsc,V_h$ so that $\sigma(V_i)=V_{i+1}$ 
for $i=1,\dotsc,r-1$ and $\sigma(V_r)=V_1$. 
Write $\overline{V}=V \otimes \overline{\F}_{\ell}$ and likewise
define $\overline{V}_i$.
Let $\overline{W}=\oplus_{i=1}^r \overline{V}_i$.
It follows that $\overline{W}$ is stable under the action of $I_\ell$. 
Let $\alpha_1 \in \overline{\F}_\ell$ be an eigenvalue for $\sigma$
acting on $\overline{W}$. By Lemma~\ref{lem:cyclic},
we know that $\alpha_2=\zeta \alpha_1$ is also an eigenvalue for $\sigma$
acting on $\overline{W}$, where $\zeta \in \overline{\F}_{\ell}$
is a primitive $r$-th root of unity (observe that this exists
as $r \le h \le d<\ell$). By Raynaud's Theorem, there exist
$n_1$, $n_2$ and characters 
$\varpi_j \; \colon \; I_\ell \rightarrow \F_{\ell^{n_j}}^\times$ such
that $\alpha_j=\varpi_j(\sigma)$. As $\sigma$ is a topological
generator for the tame inertia and the characters $\varpi_j$ are surjective,
we see that $\alpha_1$ and $\alpha_2$ have orders $\ell^{n_1}-1$, $\ell^{n_2}-1$ respectively.
Then $\zeta=\alpha_2/\alpha_1$
has order divisible by 
\[
(\ell^{n_1}-1)(\ell^{n_2}-1)/\gcd(\ell^{n_1}-1,\ell^{n_2}-1)^2 \, .
\]
Suppose first that $n_1 \ne n_2$. Without loss of generality $n_1<n_2$. Then
$\gcd(\ell^{n_1}-1,\ell^{n_2}-1) \le \ell^{n_1}-1$. Thus
\[
\frac{(\ell^{n_1}-1)(\ell^{n_2}-1)}{\gcd(\ell^{n_1}-1,\ell^{n_2}-1)^2}
 \ge
\frac{(\ell^{n_2}-1)}{(\ell^{n_1}-1)}=
\ell^{n_2-n_1}+\frac{(\ell^{n_2-n_1}-1)}{(\ell^{n_1}-1)}
>\ell. 
\]
This contradicts the fact that the order of $\zeta$ is $r<\ell$.
Thus $n_1=n_2=n$ (say). Now from Raynaud's Theorem, we know
that 
\[
\varpi_1=\psi_n^{a_0+a_1 \ell+\cdots+a_{n-1} \ell^{n-1}},
\qquad
\varpi_2=\psi_n^{b_0+b_1 \ell+\cdots+b_{n-1} \ell^{n-1}},
\]
where $\psi_n \; \colon \; I_\ell \rightarrow \F_{\ell^n}^\times$
is a fundamental character of level $n$, and $0 \le a_i,~b_i \le 1$.
Since $\psi_n(\sigma)$ has order $\ell^n-1$ and $\zeta=\varpi_2(\sigma)/\varpi_1(\sigma)$ has order $r$, we see that
\[
r \sum_{i=0}^{n-1} (a_i-b_i) \ell^i \equiv 0 
\pmod{\ell^n-1}.
\]
However $-1 \le a_i-b_i \le 1$ and so 
\[
\left\lvert r \sum_{i=0}^{n-1} (a_i-b_i) \ell^i \right\rvert \le
r \cdot (\ell^n-1)/(\ell-1).
\]
Since $r\le d \le \ell-2$, we see that $r \sum_{i=0}^{n-1} (a_i-b_i) \ell^i=0$,
and hence $a_i=b_i$ for $i=0,\dotsc,n-1$. It follows that $\zeta$
has order $1$. Since $\zeta$ is a primitive $r$-th root of unity,
we have that $r=1$. From the definition of $r$, we have that 
$\sigma(V_1)=V_1$. Similarly, $\sigma(V_j)=V_j$ for $j=2,\dots,h$.
Hence $\pi(\sigma)=1$. As we have shown that $\pi(I_\ell^w)=1$,
and as $\sigma$ is a topological generator for the tame inertia,
we have that $\pi(I_\ell)=1$, showing that $K/\Q$ is unramified
at $\ell$. This completes the proof.
\end{proof}

\section{Surjectivity for semistable principally polarized abelian threefolds}\label{sec:bound}
We now let $A/\Q$ be a principally polarized abelian threefold.
We shall make the following assumptions henceforth:
\begin{enumerate}
\item[(a)] $A$ is semistable;
\item[(b)] There is a prime $q$ such that 
the special fibre of the N\'{e}ron
model for $A$ at $q$ has toric dimension $1$.
\end{enumerate}
Let $S$ be the set of primes $q$ satisfying (b). For $q \in S$,
write $\Phi_q$ for
the group of connected components of the special
fibre of the N\'{e}ron model of $A$ at $q$. We shall suppose that
\begin{enumerate}
\item[(c)] $\ell \ge 5$;
\item[(d)] $\ell$ does not divide $\gcd(\{ q\cdot \#\Phi_q : q \in S\})$.
\end{enumerate}
Thus \cite[Section 2]{Hall11} the image of $\overline{\rho}_{A,\ell}$ contains 
a transvection. 
It follows from Theorem~\ref{thm:semistable} that $\overline{\rho}_{A,\ell}$
is either reducible or surjective. 
In this section we explain a practical method
which should in most cases produce a small integer $B$ (depending on $A$)
such that for $\ell \nmid B$, the representation $\overline{\rho}_{A,\ell}$ 
is irreducible and hence surjective.


\subsection*{Determinants of Jordan--H\"older factors}
As before $\chi \colon G_\Q \rightarrow \F_\ell^\times$ denotes the mod $\ell$ cyclotomic character.
We will study the Jordan--H\"{o}lder factors $W$ of the $G_\Q$-module $A[\ell]$.
By the determinant of such a $W$ we mean the determinant of the induced representation
$G_\Q \rightarrow \GL(W)$. 
\begin{lem}\label{lem:det}
Any Jordan--H\"{o}lder factor $W$ of the $G_\Q$-module $A[\ell]$
has determinant $\chi^r$ for some $0 \le r \le \dim(W)$.
\end{lem}
\begin{proof}
Let $W$ be such a Jordan--H\"{o}lder factor, and let 
$\psi \; \colon \; G_\Q \rightarrow \F_\ell^\times$ be its determinant.
As $J$ is semistable, for primes $p \ne \ell$,
the inertia subgroup $I_p \subset G_\Q$ acts unipotently
on $W$ and so $\psi \vert_{I_p}=1$. Moreover, 
by considering the Jordan--H\"{o}lder factors of $W$ as an $I_\ell$-module,
it follows from Raynaud's Theorem that
$\psi \vert_{I_\ell}=\chi^r \vert_{I_\ell}$ for some $0 \le r \le \dim(W)$.
Thus the character $\psi \chi^{-r}$ is unramified at all the finite
places. As the narrow class number of $\Q$ is $1$, we see
that $\psi \chi^{-r}=1$ proving the lemma.
\end{proof}

\subsection*{Weil polynomials}
For a prime $p \ne \ell$ of good reduction for $A$, we shall henceforth write
\begin{equation}\label{eqn:Pp}
P_p(x)=x^6+\alpha_p x^5+ \beta_p x^4+ \gamma_p x^3+ p \beta_p x^2 +
p^2 \alpha_p+p^3 \in \Z[x]
\end{equation}
for the characteristic polynomial of Frobenius $\sigma_p \in G_\Q$ at $p$
acting on the Tate module $T_\ell(A)$ (also known as the Weil polynomial
of $A$ mod $p$). The polynomial $P_p$ is independent
of $\ell$. 
It follows from \eqref{eqn:Pp} that the roots in $\overline{\F}_\ell$
have the form $u$, $v$, $w$, $p/u$, $p/v$, $p/w$.
\begin{lem}\label{lem:Weil}
If $P_p$ has a real root then $(x^2-p)^2$ is a factor of $P_p$.
\end{lem}
\begin{proof}
By Weil, the complex 
roots have the form $\omega_1$, $\omega_2$, $\omega_3$, $\overline{\omega}_1$, 
$\overline{\omega}_2$, $\overline{\omega}_3$ where 
$\lvert \omega_i \rvert=\sqrt{p}$ and $\overline{\omega}$
denotes the complex conjugate of $\omega$. 
Suppose $\omega_1$ is real.
Then $\omega_1=\overline{\omega}_1$ and thus
$(x-\omega_1)(x-\overline{\omega}_1)=(x-\omega_1)^2$
is a factor of $P_p$. Moreover, $\omega_1=\pm \sqrt{p}$.
The lemma follows as $P_p \in \Z[x]$.
\end{proof}

\subsection*{$1$-dimensional Jordan--H\"older factors}
Let $T$ be a non-empty set of primes of good reduction for $A$. 
Let
\begin{equation}\label{eqn:B1}
B_1(T)=\gcd(\{ p \cdot \# A(\F_p) \; : \; p \in T\}).
\end{equation}
\begin{lem}\label{lem:1d}
Suppose $\ell \nmid B_1(T)$.
The $G_\Q$-module $A[\ell]$ does not have any $1$-dimensional
or $5$-dimensional
Jordan--H\"{o}lder factors.
\end{lem}
\begin{proof}
As $\dim(A[\ell])=6$, if $A[\ell]$ has a $5$-dimensional Jordan--H\"older
factor then it has a $1$-dimensional Jordan--H\"older factor.
Suppose $W$ is a $1$-dimensional Jordan--H\"{o}lder factor of $A[\ell]$.
Then the action of $G_\Q$ on $W$ is given by a character
$\psi \; \colon \; G_\Q \rightarrow \F_\ell^\times$. It follows
from Lemma~\ref{lem:det} that $\psi=1$ or $\chi$.

Let $p$ be a prime of good reduction for $A$, and suppose
$\ell \ne p$.
Thus
$P_p$ has root $\overline{1}$ or $\chi(\sigma_p)=\overline{p}
\in \F_\ell$.
Since the roots of $P_p$ have the form $u$, $v$, $w$,
$p/u$, $p/v$, $p/w$, we know in either case
that $\overline{1}$ is a root, and so
\[
\# A(\F_p)=P_p(1) \equiv 0 \pmod{\ell}.
\]
Thus if $p$ is a prime of good reduction for $A$, then $\ell$
divides $p \cdot \#A(\F_p)$. This proves the lemma.
\end{proof}
Since $\#A(\F_p)>0$, we have $B_1(T) \ne 0$, and so we can always rule out
$1$-dimensional
and $5$-dimensional factors for large $\ell$.

\subsection*{$2$-dimensional Jordan--H\"older factors}
\begin{lem}\label{lem:2d}
Suppose the $G_\Q$-module $A[\ell]$ does not have any $1$-dimensional
Jordan--H\"{o}lder factors, but has either a $2$-dimensional or $4$-dimensional
irreducible subspace $U$. Then $A[\ell]$ has a $2$-dimensional Jordan--H\"{o}lder
factor $W$ with determinant $\chi$.
\end{lem}
\begin{proof}
Suppose $\dim(U)=2$. 
If the restriction of the Weil pairing to $U$ is non-degenerate then $\det(U)=\chi$
and we can take $W=U$. Thus we may suppose that the restriction of the Weil pairing
to $U$ is degenerate.
Thus $U \cap U^\perp \ne 0$,
where 
\[
U^\perp=\{ v \in A[\ell] \; : \; 
\text{$\langle v , u \rangle=0$ for all $u \in U$}\}.
\]
The Galois invariance of the Weil-pairing implies that $U^\perp$
is a $G_\Q$-submodule of $A[\ell]$. Since $U$ is irreducible and 
$U \cap U^\perp \ne 0$ we have that $U \subseteq U^\perp$.
However $U^\perp$ is $4$-dimensional. Thus
each of the $2$-dimensional quotients in the sequence
$0 \subset U \subset U^\perp \subset A[\ell]$ is $2$-dimensional,
must be irreducible (as $A[\ell]$ does not have $1$-dimensional
factors) and has
determinant $1$ or $\chi$ or $\chi^2$ by Lemma~\ref{lem:det}. 
Since $\det(A[\ell])=\chi^3$ we see that one the three quotients must have determinant $\chi$.
This completes the proof for the case $\dim(U)=2$.

Now suppose that $\dim(U)=4$.
If the restriction of the Weil pairing to $U$ is degenerate,
then $U \subseteq U^\perp$ as before; this is impossible as $\dim(U^\perp)=2$.
It follows that the restriction of the Weil pairing to $U$ is non-degenerate
and so $\det(U)=\chi^2$. As $\det(A[\ell])=\chi^3$, we have
that $A[\ell]/U$ is an irreducible $2$-dimensional
$G_\Q$-module with determinant $\chi$. This completes the proof.
\end{proof}

Let $N$ be the conductor of $A$. Let $W$ be a $2$-dimensional Jordan--H\"older
factor of $A[\ell]$ with determinant $\chi$. 
The representation 
\[
\tau \colon G_\Q \rightarrow \GL(W) \cong \GL_2(\F_\ell)
\]
is odd (as the determinant is $\chi$),
irreducible (as $W$ is a Jordan--H\"older factor) and $2$-dimensional. 
By Serre's modularity conjecture, now a theorem of Khare and Wintenberger
\cite[Theorem~1.2]{kw}, this representation arises from a newform $f$ of level $M \mid N$
and weight $2$. Let $\OO_f$ be the ring of integers of the number field generated by the
Hecke eigenvalues of $f$. Then there is a prime $\lambda \mid \ell$ of $\OO_f$ such that
for all primes $p \nmid \ell N$, 
\[
\tr(\tau(\sigma_p)) \equiv c_p(f) \pmod{\lambda}
\]
where $\sigma_p \in G_\Q$ is a Frobenius element at $p$ and $c_p(f)$
is the $p$-th Hecke eigenvalue of $f$. Hence $x^2-c_p(f) x+ p$
is congruent modulo $\lambda$ to the characteristic polynomial of $\tau(\sigma_p)$.
As $W$ is a Jordan--H\"older factor of $A[\ell]$ we see that $x^2-c_p(f)x+p$
is a factor modulo $\lambda$ of $P_p$. Now let $H_{M,p}$ be the $p$-th Hecke polynomial
for the new subspace $S_2^\mathrm{new}(M)$ of cusp forms of weight $2$ and level $M$.
This has the form $H_{M,p}=\prod (x-c_p(g))$ where $g$ runs through the 
newforms of weight $2$ and level $M$. We shall write
\[
H^\prime_{M,p}(x)=x^d H_{M,p}(x+p/x) \in \Z[x], \qquad d=\deg(H_{M,p})=\dim(S_2^\mathrm{new}(M)) \, .
\]
It follows that $x^2-c_p(f)x+p$ divides $H^\prime_{M,p}$. Let 
\begin{equation}\label{eqn:RMp}
R(M,p)=\Res(P_p,H^\prime_{M,p}) \in \Z,
\end{equation}
where $\Res$ denotes resultant.
It is immediate that $\lambda \mid R(M,p)$.
As $R(M,p)$ is a rational integer, we have $\ell \mid R(M,p)$.
If $R(M,p) \ne 0$ then we obtain a bound on $\ell$.
We can of course work directly with $\Res(P_p, x^2-c_p(f)x+p)$,
which produces an integer in $\OO_f$ divisible by $\lambda$,
and if this algebraic integer is non-zero it would lead us to a bound on $\ell$.
However, in general it is much easier and faster to write down the Hecke polynomials $H_{M,p}$
than it is to compute the individual eigenforms $f$.

The integers $R(M,p)$ can be very large (see the example below). Given a non-empty
set $T$ of rational primes $p$ of good reduction for $A$, we shall let
\[
R(M,T)=\gcd(\{ p \cdot R(M,p) \; : \; p \in T\}).
\]
In practice, we have found that for a suitable choice of $T$, the value $R(M,T)$
is fairly small.
Now let
\[
B^\prime_2(T)=\lcm(R(M,T))
\]
where $M$ runs through the divisors of $N$ such that $\dim(S_2^\mathrm{new}(M)) \ne 0$,
and let
\[
B_2(T)=\lcm(B_1(T),B^\prime_2(T)),
\]
where $B_1(T)$ is given by \eqref{eqn:B1}.

\begin{lem}\label{lem:Serre}
Let $T$ be a non-empty set of rational primes of good reduction for $A$,
and suppose $\ell \nmid B_2(T)$. Then $A[\ell]$ does not have $1$-dimensional
Jordan--H\"{o}lder factors, and does not have irreducible $2$- or $4$-dimensional
subspaces.
\end{lem}
\begin{proof}
By Lemmas~\ref{lem:1d} and ~\ref{lem:2d} it is enough to rule
out the existence of a $2$-dimensional Jordan--H\"older factor with character $\chi$.
This follows from the above discussion.
\end{proof}
Of course we fail to bound $\ell$ in the above lemma if $R(M,p)=0$
for all primes $p$ of good reduction. Here are two situations where
this can happen:
\begin{itemize}
\item Suppose $A$ is isogenous over $\Q$ to $E \times A^\prime$ where
$E$ is an elliptic curve and $A^\prime$ an abelian surface. If we take $M \mid N$ to be the conductor of the
elliptic curve, and $f$ to be the newform associated to $E$ by modularity,
then $x^2-c_p(f)x+p$ is a factor of $P_p(x)$ in $\Z[x]$. 
Thus the resultant $R(M,p)=0$ for all $p \nmid N$.
\item Suppose the abelian threefold $A$ is of $\GL_2$-type. It is therefore modular
by Khare and Wintenberger \cite{kw}, and if we let $f$ be the corresponding eigenform,
then again $x^2-c_p(f)x+p$ is a factor of $P_p(x)$ in $\OO_f[x]$, and so the resultant $R(M,p)=0$ for all $p \nmid N$.
\end{itemize}
Note that in both these situations $\End_{\overline{\Q}}(A) \ne \Z$. 
We expect, but are unable to prove, that if $\End_{\overline{\Q}}(A) =\Z$
then there will be primes $p$ such that $R(M,p) \ne 0$.

\subsection*{$3$-dimensional Jordan--H\"older factors}

\begin{lem}\label{lem:3d}
Suppose $A[\ell]$ has Jordan--H\"{o}lder
filtration $0 \subset U \subset A[\ell]$ where both $U$ and $A[\ell]/U$
are irreducible and $3$-dimensional.
Moreover, let $u_1$, $u_2$, $u_3$ be a basis for $U$, and let
\[
G_\Q \rightarrow \GL_3(\F_\ell), \qquad
\sigma \mapsto M(\sigma)
\]
give the action of $G_\Q$ on $U$ with respect to this basis.
Then we can extend $u_1$, $u_2$, $u_3$ to a symplectic basis $u_1$, $u_2$, $u_3$,
$w_1$, $w_2$, $w_3$ for $A[\ell]$ so that the action of $G_\Q$
on $A[\ell]$ with respect to this basis is given by
\[
G_\Q \rightarrow \GSp_6(\F_\ell), \qquad 
\sigma \mapsto 
\left(
\begin{array}{c|c}
M(\sigma) & * \\
\hline
\mathbf{0} & \chi(\sigma) (M(\sigma)^t)^{-1} \\
\end{array}
\right) \, .
\]
\end{lem}
\begin{proof}
Any bilinear alternating
pairing on an odd dimensional space (in characteristic $\ne 2$)
must be degenerate. We thus deduce that $U \subseteq U^\perp$
as in the proof of Lemma~\ref{lem:2d}.
As both spaces have dimension $3$, we have $U=U^\perp$.
Let $u_1$, $u_2$, $u_3$ be a basis for $U$. Then
$\langle u_i , u_j \rangle=0$. Extend this to a symplectic
basis $u_1$, $u_2$, $u_3$, $w_1$, $w_2$, $w_3$ for $A[\ell]$:
meaning that in addition to $\langle u_i , u_j \rangle=0$
the basis satisfies
$\langle w_i, w_j \rangle=0$, and $\langle u_i, w_j \rangle =
\delta_{i,j}$ where $\delta_{i,j}$ is the Kronecker delta. The 
lemma follows from the identity
$\langle u_i, \sigma w_j \rangle= 
\chi(\sigma) \langle \sigma^{-1} u_i, w_j \rangle$ for $\sigma \in G_\Q$.
\end{proof}
Now let $U$ be as in Lemma~\ref{lem:3d}. By Lemma~\ref{lem:det},
we have that $\det(U)=\chi^r$ and $\det(A[\ell]/U)=\chi^s$ 
where $0 \le r$, $s \le 3$. Moreover, as $\det(A[\ell])=\chi^3$
we have that $r+s=3$. 

\begin{lem}\label{lem:case1}
Let $p$ be a prime of good reduction for $A$. For ease
write $\alpha$, $\beta$ and $\gamma$ for 
the coefficients $\alpha_p$, 
$\beta_p$, $\gamma_p$ in \eqref{eqn:Pp}. Suppose $p+1 \ne \alpha$
(this is certainly true for $p \ge 36$ as $\lvert \alpha \rvert \le 6 \sqrt{p}$).
Let
\begin{equation}\label{eqn:de}
\delta=
\frac{- p^2\alpha + p^2 + p\alpha^2 - p\alpha - p\beta 
    + p - \beta + \gamma}{
(p-1)(p+1-\alpha)  
} \in \Q,
\qquad
\epsilon= \delta+\alpha \in \Q.
\end{equation}
Let
\begin{equation}\label{eqn:g}
g(x)=(x^3+\epsilon x^2+\delta x-p)(x^3-\delta x^2-p\epsilon x-p^2) \in \Q[x]. 
\end{equation}
Write $k$ for the greatest common divisor of the numerators of 
the coefficients in $P_p-g$. Let 
\[
K_p=p(p-1)(p+1-\alpha) k.
\]
Then $K_p \ne 0$. Moreover, if 
$\ell \nmid K_p$ then $A[\ell]$ does not have a Jordan--H\"older
filtration as in Lemma~\ref{lem:3d} with $\det(U)=\chi$ or $\chi^2$.
\end{lem}
\begin{lem}\label{lem:case2}
Let $p$ be a prime of good reduction for $A$. 
Write $\alpha$, $\beta$ and $\gamma$ for 
the coefficients $\alpha_p$, 
$\beta_p$, $\gamma_p$ in \eqref{eqn:Pp}. Suppose $p^3+1 \ne p\alpha$
(this is true for $p \ge 5$ as $\lvert \alpha \rvert \le 6 \sqrt{p}$).
Let
\begin{equation}\label{eqn:ded}
\delta^\prime=
\frac{
-p^5\alpha + p^4 + p^3\alpha^2 - p^3\beta - p^2\alpha + p\gamma + p - \beta
 }{
(p^3-1)(p^3+1-p\alpha)  
} \in \Q,
\qquad
\epsilon^\prime= p\delta^\prime+\alpha \in \Q.
\end{equation}
Let
\begin{equation}\label{eqn:gd}
g^\prime(x)=
(x^3+\epsilon^\prime x^2+\delta^\prime x-1)(x^3-p\delta^\prime x^2-p^2 \epsilon^\prime x-p^3) \in \Q[x] .
\end{equation}
Write $k^\prime$ for the greatest common divisor of the numerators of 
the coefficients in $P_p-g^\prime$. Let 
\[
K_p^\prime=p(p^3-1)(p^3+1-p\alpha) k^\prime.
\]
Then $K_p^\prime \ne 0$. Moreover, if 
$\ell \nmid K_p^\prime$ then $A[\ell]$ does not have a Jordan--H\"older
filtration as in Lemma~\ref{lem:3d} with $\det(U)=1$ or $\chi^3$.
\end{lem}
\begin{proof}[Proofs of Lemma~\ref{lem:case1} and~\ref{lem:case2}]
For now let $p$  be a prime of good reduction for $A$,
and suppose that $\ell \ne p$.
Suppose $A[\ell]$ has a Jordan--H\"older filtration 
$0 \subset U \subset A[\ell]$ where $U$ and $A[\ell]/U$
are $3$-dimensional (i.e.\ as in Lemma~\ref{lem:3d}).
Then $\det(U)=\chi^r$ with $0 \le r \le 3$.
Let $\sigma_p \in G_\Q$ denote a Frobenius element at $p$.
Let $M=M(\sigma_p)$ as in Lemma~\ref{lem:3d}. Then $\det(M)=\overline{p}^r \in \F_\ell$.
Moreover, from the lemma,
\[
P_p(x) \equiv \det(xI-M)\det(xI-p M^{-1}) \pmod{\ell}.
\]
Write
\[
\det(xI-M) \equiv x^3+u x^2+v x - p^r \pmod{\ell}.
\]
Then
\begin{align*}
\det(xI-p M^{-1})
& =-p^{-r} \cdot x^3 \cdot \det(p x^{-1} I - M) \\
& \equiv x^3-p^{1-r} v x^2-p^{2-r} u x-p^{3-r} \pmod{\ell}. 
\end{align*}
Let
\[
a=\begin{cases}
u & \text{if $r=0$ or $1$}\\
-p^{-1} v & \text{if $r=2$}\\
-p^{-2} v & \text{if $r=3$}
\end{cases}
\qquad
b=\begin{cases}
v & \text{if $r=0$ or $1$}\\
-u & \text{if $r=2$}\\
-p^{-1} u & \text{if $r=3$}.
\end{cases}
\]
If $r=1$ or $2$ then 
\begin{equation}\label{eqn:case1}
P_p(x) \equiv (x^3+a x^2+bx-p)(x^3-bx^2-pa x-p^2) \pmod{\ell}.
\end{equation}
If $r=0$ or $3$ then
\begin{equation}\label{eqn:case2}
P_p(x) \equiv (x^3+ax^2+bx-1)(x^3-pb x^2-p^2 a x-p^3) \pmod{\ell}.
\end{equation}
We now suppose that $\ell \nmid K_p$
and prove Lemma~\ref{lem:case1} which corresponds to $r=1$ or $2$.
We thus suppose that \eqref{eqn:case1} holds. 
Comparing the 
coefficients of $x^5$ in \eqref{eqn:case1} we have that
$a \equiv b+\alpha \pmod{\ell}$. Substituting this into \eqref{eqn:case1}
and comparing the coefficients of $x^4$ and $x^3$ we obtain
\begin{align*}
b^2  + (p + \alpha - 1)\cdot b  + (p\alpha + \beta)  \; & \equiv \; 0 \pmod{\ell}\\
(p + 1)\cdot b^2  + 2 p \alpha \cdot b  + (p^2 + p\alpha^2 + p+ \gamma) \;
 & \equiv \;  0 \pmod{\ell}. 
\end{align*}
Eliminating $b^2$ we obtain the following congruence which is linear in $b$:
\[
-(p-1)(p+1-\alpha) \cdot b +
(- p^2\alpha + p^2 + p\alpha^2 - p\alpha - p\beta 
    + p - \beta + \gamma )
\equiv 0 \pmod{\ell}.
\]
As $\ell \nmid K_p$ we have
$\ell \nmid (p-1)(p+1-\alpha)$, and so we can solve for $b$ mod $\ell$.
It follows that $b \equiv \delta$ and $a \equiv b+\alpha \equiv \epsilon
\pmod{\ell}$ where $\delta$ and $\epsilon$ are given by \eqref{eqn:de}.
Substituting into \eqref{eqn:case1}, we see that
$P_p \equiv g \pmod{\ell}$ where $g$ is given by \eqref{eqn:g}.
Thus $\ell$ divides the greatest common divisor of the numerators
of the coefficients of $P_p-g$ showing that $\ell \mid k$ (in the notation
of Lemma~\ref{lem:case1}). As $k \mid K_p$ and $\ell \nmid K_p$
we obtain a contradiction. Thus if $\ell \nmid K_p$ then
$A[\ell]$ does not have a Jordan--H\"{o}lder filtration 
as in Lemma~\ref{lem:3d} with $\det(U)=\chi$ or $\chi^2$.

We need to show that $K_p \ne 0$. We are supposing that
$p+1 \ne \alpha$ thus we need to show that $P_p \ne g$.
Suppose $P_p=g$. 
As $g$ is the product of two cubic
polynomials, it follows that $P_p$ has at least two real roots.
By Lemma~\ref{lem:Weil}, we see that $(x^2-p)^2 \mid P_p$.
It follows that $P_p=g$ must have two rational roots.
Since all the roots have absolute value $\sqrt{p}$,
this is a contradiction. We deduce that $K_p \ne 0$
as required. This completes the proof of Lemma~\ref{lem:case1}.
The proof of Lemma~\ref{lem:case2} is practically identical.
\end{proof}

\subsection*{Summary}
The following theorem summarizes Section~\ref{sec:bound}. 
\begin{thm}
Let $A$ and $\ell$ satisfy conditions (a)--(d)
at the beginning of Section~\ref{sec:bound}.
Let $T$ be a non-empty set of primes of good reduction for $A$.  
Let
\[
B_3(T)=\gcd(\{K_p \; : \; p \in T\}), \qquad
B_4(T)=\gcd(\{K^\prime_p \; : \; p \in T\}),
\]
where $K_p$ and $K^\prime_p$ are defined in Lemmas~\ref{lem:case1} and~\ref{lem:case2}.
Let
\[
B(T)=\lcm(B_2(T),B_3(T),B_4(T))
\]
where $B_2(T)$ is as in Lemma~\ref{lem:Serre}. If $\ell \nmid B(T)$ then
$\overline{\rho}_{A,\ell}$ is surjective.
\end{thm}
\begin{proof}
By Theorem~\ref{thm:semistable} we know that $\overline{\rho}_{A,\ell}$
is either reducible or surjective. Lemmas~\ref{lem:Serre},~\ref{lem:case1} and~\ref{lem:case2}
ensure that $\overline{\rho}_{A,\ell}$ cannot be reducible. Hence it must be surjective.
\end{proof}

\section{Proof of Corollary~\ref{cor:surjective}}
We implemented the method described in Section~\ref{sec:bound} in \texttt{Magma}~\cite{Magma}.
The model given in \eqref{eqn:gen3} for the curve $C$ has good reduction at $2$.
Let $J$ be the Jacobian of $C$. This has conductor $N=8907=3 \times 2969$ 
(the algorithm used by \texttt{Magma}
for computing the conductor is described in \cite{doccond}). 
As $N$ is squarefree, the Jacobian $J$ is semistable.
Completing the square
in \eqref{eqn:gen3} we see that the curve $C$ has the following \lq simplified\rq\
Weierstrass model.
\[
y^2 = x^8 + 2x^7 + 5x^6 + 6x^5 + 4x^4 + 2x^3
    + x^2 + 2x + 1.
\]
Denote the polynomial on the right-hand side by $f$. Then
\[
f \equiv (x + 1)(x+2)^2 (x^2 + x + 2 ) (x^3 + 2x^2 + 2x + 2) \pmod{3}
\]
and
\[
f \equiv (x + 1) (x + 340) (x + 983)^2 (x^2 + x + 1)(x^2 + 663 x + 1350) \pmod{2969}.
\]
Here the non-linear factors in both factorizations are irreducible. As $f$ has precisely
one double root in $\overline{\F}_3$ and one double root in $\overline{\F}_{2969}$
with all other roots simple, we see that the N\'eron models for $J$ at $3$ and $2969$
have special fibres with toric dimension $1$. 
We found that $\# \Phi_{3}=\# \Phi_{2969}=1$.
Thus the image 
of $\overline{\rho}_{J,\ell}$ contains a transvection for all $\ell \ge 3$.

We now suppose $\ell \ge 5$. 
By Theorem~\ref{thm:semistable} we know that $\overline{\rho}_{J,\ell}$
is either reducible or surjective.
In the notation of Section~\ref{sec:bound}, we  
take our chosen set of primes of good reduction to be 
$T=\{2,5,7\}$.
We note that
\[
\# J(\F_2)=2^5, \qquad \# J(\F_5)=2^7, \qquad \# J(\F_7)=2^6 \times 7.
\]
It follows from Lemma~\ref{lem:1d} that $J[\ell]$ does not have $1$- or $5$-dimensional Jordan--H\"{o}lder factors.
Next we consider the existence of $2$- or $4$-dimensional irreducible subspaces. The possible
values $M \mid N$ such that $S_2^{\mathrm{new}}(M) \ne 0$ are $M=2969$ and $M=8907$, where the dimensions are $247$ and $495$ respectively.
Unsurprisingly, the resultants $R(M,p)$ (defined in \eqref{eqn:RMp})
are too large to reproduce here.
For example, we indicate that $R(8907,7) \sim 1.63\times 10^{2344}$. However,
\[
R(M,T)=\gcd\left( 2 \cdot R(M,2), \;
5 \cdot R(M,5),\; 7 \cdot R(M,7) \right)=
\begin{cases} 
2^4 & M=2969 , \\
2^{22} & M=8907 .
\end{cases}
\]
It follows from Lemma~\ref{lem:Serre} that $J[\ell]$ does not have
$2$- or $4$-dimensional irreducible subspaces. It remains to eliminate
the possibility of a Jordan--H\"older filtration $0 \subset U \subset J[\ell]$
where both $U$ and $J[\ell]/U$ are $3$-dimensional.
In the notation of Lemma~\ref{lem:case1},
\[
K_2=14,
\qquad K_5=6900,
\qquad K_7=83202 .
\]
Then $\gcd(K_2,K_5,K_7)=2$. Lemma~\ref{lem:case1} eliminates the
case where $\det(U)=\chi$ or $\chi^2$. Moreover,
\[
K_2^\prime=154490,\qquad
K_5^\prime=15531373270380,\qquad
K_7^\prime=10908656905042386.
\]
Then $\gcd(K_2^\prime,K_3^\prime,K_7^\prime)=2$.
Lemma~\ref{lem:case2} eliminates the case where $\det(U)=1$
or $\chi^3$. It follows that $\overline{\rho}_{J,\ell}$ is
irreducible and hence surjective for all $\ell \ge 5$.

\bigskip

It remains to show that $\overline{\rho}_{J,3}$ is surjective.
Denote $\overline{\rho}=\overline{\rho}_{J,3}$.
Write $G=\overline{\rho}(G_\Q)$. For a prime $p$ of good reduction,
let $\sigma_p \in G_\Q$ denote a Frobenius element at $p$ and 
$\overline{P}_p \in \F_3[t]$ be 
the characteristic polynomial of $\sigma_p$ acting
on $J[3]$. Let $N_p$ be the multiplicative order of the image
of $t$ in the algebra $\F_{3}[t]/\overline{P}_p$.
It is immediate that $N_p$ divides the order of $\overline{\rho}(\sigma_p)$
and hence divides the order of $G$. We computed
\[
N_2=2^3 \times 5, \qquad N_5=2 \times 13, \qquad
N_{19}=7, \qquad
N_{37}=2 \times 3^2.
\]
Thus the order of $G$ is divisible by $2^3 \times 3^2 \times 5 \times 7 \times 13$. We checked that the only subgroups of $\GSp_6(\F_3)$ with order divisible
by this are $\Sp_6(\F_3)$ and  $\GSp_6(\F_3)$. 
As the mod $3$ cyclotomic character is surjective on $G_\Q$, we have that
$G=\GSp_6(\F_3)$. This completes the proof of the corollary.

\bibliographystyle{annotate}
\bibliography{biblio}

\end{document}